\documentclass[11pt]{amsart}
\usepackage{fullpage}
\usepackage{url}
\usepackage{amsfonts}
\usepackage{amsmath}
\usepackage{amssymb}
\usepackage{amsthm}
\usepackage{verbatim}
\usepackage{graphicx}

\newtheorem{theorem}{Theorem}

\newtheorem{lemma}[theorem]{Lemma}

\def\RR{\mathbb{R}}
\def\rank{\operatorname{rank}}
\numberwithin{theorem}{section}
\numberwithin{equation}{section}
\begin{document}

\title{Typical Real Ranks of Binary Forms}
\author{Grigoriy Blekherman}
\maketitle
\begin{abstract}
We prove a conjecture of Comon and Ottaviani that typical real  Waring ranks of bivariate forms of degree $d$ take all integer values between $\lfloor \frac{d+2}{2}\rfloor$ and $d$. That is we show that for all $d$ and all $\lfloor \frac{d+2}{2}\rfloor \leq m \leq d$ there exists a bivariate form $f$ such that $f$ can be written as a linear combination of $m$ $d$-th powers of real linear forms and no fewer, and additionally all forms in an open neighborhood of $f$ also possess this property. Equivalently we show that for all $d$ and any $\lfloor \frac{d+2}{2}\rfloor \leq m \leq d$ there exists a symmetric real bivariate tensor $t$ of order $d$ such that $t$ can be written as a linear combination of $m$ symmetric real tensors of rank $1$ and no fewer, and additionally all tensors in an open neighborhood of $t$ also possess this property.
\end{abstract}

\section{Introduction}

Symmetric tensor decomposition (also known as the Waring problem) is usually studied over the complex numbers: given a multivariate form $f \in \mathbb{C}[x_1,\dots,x_n]$ of degree $d$ the problem asks to find the least integer number $m$ such that $f=\sum_{i=1}^m \ell_i^d$, where $\ell_i$ are linear forms. The number $m$ is known as the \textit{Waring rank} or \textit{symmetric tensor rank} of $f$. It is of significant interest to find efficient algorithms to compute a minimal representation of $f$ \cite{OO}. The symmetric tensor decomposition problem over the complex numbers has been widely investigated \cite{AH},\cite{MonomialSolution},\cite{ChiantiniCiro},\cite{LandsbergTeitler},\cite{Mella}. 

There is a unique generic rank forms forms in $\mathbb{C}[x_1,\dots,x_n]_d$, which depends on the degree $d$ and the number of variables $n$. Alexander-Hirschowitz theorem gives the generic rank for all $n$ and $d$ \cite{AH}. The question of uniqueness, and more generally, the variety of all possible decompositions for a form $f$ has also been studied extensively \cite{ChiantiniCiro},\cite{Mella}.

The same question can be asked over the real numbers, which often makes more sense with respect to potential applications. Given a multivariate form $f \in \mathbb{R}[x_1,\dots,x_n]$ of degree $d$ we now ask to find the least integer number $m$ such that $f=\sum_{i=1}^m \ell_i^d$ where $\ell_i\in \RR[x_1,\dots,x_n]$ are \textit{real} linear forms. The theory of real symmetric tensor decompositions is not nearly as developed. One obstacle is that there may be more than one ``generic" rank. Following Comon and Ottaviani in \cite{OC} we will call any rank that occurs on an open subset of $\mathbb{R}[x_1,\dots,x_n]_d$ (with the Euclidean topology induced by viewing $f$ as a vector of coefficients) a typical rank.

The complex tensor decomposition problem in two variables has been already well-understood by Sylvester and completely and algorithmically solved by Comas and Seiguer in \cite{Comas}. In contrast for binary forms over $\RR$ the list of all typical ranks was conjectured by Comon and Ottaviani \cite{OC}. In this paper we prove this conjecture and solve the problem of finding all typical ranks for real binary forms. It has already been observed by Reznick that all ranks $m$ with $1 \leq m \leq d$ occur as real ranks of binary in $\mathbb{R}[x,y]_d$ \cite{RezInertia}. It was noted by Comon and Ottaviani that only ranks $m$ with $\lfloor\frac{d+2}{2}\rfloor \leq m \leq d$ may be typical ranks for forms in $\mathbb{R}[x,y]_d$. They showed that all such ranks occur for $d\leq 5$ and later this was proved by Ballico for $d \leq 7$ in \cite{Ballico}. It was also shown in \cite{OC} that ranks $\lfloor\frac{d+2}{2}\rfloor $ and $d$ are typical. The case of real ranks of bivariate monomials was completely analyzed in by Boij, Carlini and Geramita in \cite{Monomials}. The case of Waring ranks of binary forms over some field extensions of $\RR$ was analyzed by Reznick in \cite{RezBinaryLength}. Finally we note that the question of decompositions of real forms $f$ of even degree as sums of $2d$-th powers of linear forms is related to the truncated moment problem in real analysis \cite{LasserreBook}, \cite{Rez3}.

\section{Proofs}

Let $f \in \RR[x,y]_d$ be given by $f=\sum_{i=0}^d a_ix^iy^{d-i}$. We can associate to $f$ the differential operator $\partial f$ given by $$\partial f=\sum_{i=0}^d a_i\frac{\partial^d}{\partial x^i\partial y^{d-i}}.$$ The \textit{apolar ideal} of $f$ denoted by $f^\perp$ is the ideal of all forms in $\RR[x,y]$ whose differential operator annihilates $f$:
$$f^\perp=\{h \in \RR[x,y] \,\mid \, \partial h(f)=0\}.$$

The following is a special bivariate version of the general Apolarity Lemma \cite{IarrKanev}. It was already known to Sylvester \cite{RezInertia}.
\begin{lemma}[Apolarity Lamma] \label{LEMMA apolarity} Let $f \in \RR[x,y]_d$ be a bivariate form of degree $d$. The form $f$ can be written as a linear combination of $d$-th powers of linear forms
$$f=\sum_{i=1}^rc_i(a_ix+b_iy)^d,$$
if and only if the form $p=(b_1x-a_1y)\cdots (b_rx-a_ry)$ is in the apolar ideal $f^{\perp}$.
 \end{lemma}

From the Apolarity Lemma we see that $\rank f=m$ if and only if $m$ is the lowest degree for which $f^\perp$ contains a form with all roots real and distinct.
The structure of bivariate apolar ideals is highly regular. It was already shown by Sylvester that all bivariate apolar ideals are complete (empty) intersections and the converse also holds.

\begin{theorem}\label{LEMMA degrees}
Let $f \in \mathbb{R}[x,y]_d$ then $f^{\perp}$ is a complete intersection ideal over $\mathbb{C}$, i.e $f^\perp$ is generated by two real forms $g_1,g_2$ such that $\deg g_1+\deg g_2=d+2$ and $\mathcal{V}_{\mathbb{C}}(g_1,g_2)=\emptyset$. Conversely, any two such forms $g_1,g_2$ generate an ideal $f^{\perp}$ for some $f\in \RR[x,y]$ of degree $\deg g_1+\deg g_2-2$.
\end{theorem}

It is well known that the forms $f \in \RR[x,y]_d$ for which the generator degrees of $f^{\perp}$ are not equal to $\left(\frac{d+2}{2},\frac{d+2}{2}\right)$ if $d$ is even or $\left(\frac{d+1}{2},\frac{d+3}{2}\right)$ if $d$ is odd lie on a Zariski closed subset of $\RR[x,y]_d$. This can be seen for instance by considering the middle catalecticant matrix of $f$ and observing that these generator degrees occur precisely when the middle catalecticant matrix is not of full rank \cite{RezInertia}.  
Therefore when $f^{\perp}$ is generated by forms of degrees $\left(\frac{d+2}{2},\frac{d+2}{2}\right)$ if $d$ is even or $\left(\frac{d+1}{2},\frac{d+3}{2}\right)$ if $d$ is odd we will say that $f^\perp$ is generated \textit{in generic degrees}.

We now observe that an essential obstruction to $f$ being a typical form of rank $m$ is that $(f^\perp)_{m-1}$ may contain a form with all real roots, but no forms with all real \textit{distinct} roots. Then perturbing $f$ may result in a form of rank $m-1$. As the next Lemma shows, this is the only obstruction to $f$ being typical if $f^\perp$ is generated in generic degrees.

\begin{lemma}\label{LEMMA generic}
Let $f \in \mathbb{R}[x,y]_d$ such that $f^{\perp}$ is generated in generic degrees and let $m=\rank f$. Then $f$ is a typical form if and only if all forms in $(f^\perp)_{m-1}$ have (at least a conjugate pair of) complex roots.\end{lemma}
\begin{proof}

First suppose that $(f^\perp)_{m-1}$ contains a form with all real roots, call this form $s$. For any $\epsilon >0$ there exists a form $r_\epsilon \in \RR[x,y]_{m-1}$ in the $\epsilon$-neighborhood of $s$ such that $s+r(\epsilon)$ has distinct real zeroes. Let $h_\epsilon=\partial (s+r_\epsilon)(f)=\partial r_\epsilon(f)$ and consider the map $T_{\epsilon}: \RR[x,y]_d \rightarrow \RR[x,y]_{d-m+1}$ given by $T(f)=\partial (s+r_{\epsilon})(f)$.  Since $T_{\epsilon}$ is a linear map, which for small enough $\epsilon$ is sufficiently close to just applying $\partial s$, we can find a form $g_{\epsilon} \in \RR[x,y]_d$ such that $T_{\epsilon}(g_{\epsilon})=-h_{\epsilon}$ and $g_{\epsilon}$ approaches $0$ as $\epsilon$ goes to $0$. Then we see that $\partial (s+r_{\epsilon})(f+g_{\epsilon})=0$ and by Apolarity Lemma there exist forms of rank at most $m-1$ in any $\epsilon$-neighborhood of $f$, which is a contradiction.

For the other direction, we note that all forms $h$ in a sufficiently small neighborhood of $f$ will have $h^{\perp}$ generated in generic degrees. In this neighborhood of $f$ the ideal $h^{\perp}$ depends continuously of the coefficients of $h$. This can be seen, for example, by noting that $h^{\perp}$ consists of the forms in the kernels of the catalecticant matrices of $h$ \cite{RezInertia}. As long as $h^{\perp}$ is generated in generic degrees the dimensions of the kernels do not change and therefore there is a continuous dependence. 
Now perturb the coefficients of $f$ slightly and call the resulting form $h$. For a small enough perturbation, given the continuos dependence of $h^{\perp}$ on the coefficients of $h$ we can ensure that $(h^{\perp})_{m-1}$ has no forms with all real roots, while $(h^{\perp})_m$ has such a form. Thus the rank of $h$ is $m$, where $h$ is any sufficiently small perturbation of the coefficients of $f$. 
\end{proof}

\begin{theorem}
All ranks $m$ with $\lfloor\frac{d+2}{2}\rfloor \leq m \leq d$ are typical for forms in $\mathbb{R}[x,y]_d$.
\end{theorem}
\begin{proof}
We use induction on the degree $d$. The base case $d=2$ is just bivariate quadratic forms and the real rank corresponds to the usual rank of the matrix. Therefore there is only one typical rank, which is $2$.

Inductive Step: $d \implies d+1$. We first note that it was already shows in \cite{OC} that rank $d+1$ is typical for forms in $\mathbb{R}[x,y]_{d+1}$. Suppose that $f \in \mathbb{R}[x,y]_d$ is a typical form of rank $\lfloor \frac{d+3}{2} \rfloor \leq m \leq d$. By perturbing $f$ we may assume that the apolar ideal $f^\perp$ is generated in generic degrees.

Suppose $d=2k$ is even. Then $f^{\perp}$ is generated by forms $p_1,p_2$ with $\deg p_1=\deg p_2=k+1$. First suppose that $m=k+1$. We may choose a generator $p_1 \in (f^\perp)_m$ such that $p_1$ has all real distinct roots and let $p_2$ be a form in $(f^\perp)_m$ linearly independent from $p_1$. Now let $\ell \in \RR[x,y]_1$ be any linear form such that the zero of $\ell$ is not one of the zeroes of $p_1$ and consider the ideal $I=\langle p_1,\ell p_2 \rangle$. The forms $p_1$ and $\ell p_2$ form a complete intersection over $\mathbb{C}$. By Theorem \ref{LEMMA degrees}, $I$ is the apolar ideal of some form $g \in \RR[x,y]_{d+1}$. Since we have $g^{\perp} \subset f^\perp$ By Lemma \ref{LEMMA generic} we know that $g$ is a typical form of rank $m$.

Now suppose that $m >k+1$. By Apolarity Lemma there exists $s\in (f^\perp)_m$ such that $s$ has all real distinct roots and by Lemma \ref{LEMMA generic} we know that all forms in $(f^\perp)_{m-1}$ have at least $2$ complex roots.   Since $s \in (f^\perp)_m$ we can write  $s=p_1q_1+p_2q_2$ for $q_1,q_2 \in \RR[x,y]_{m-k-1}$. We now claim that we may choose two generators $p_1$ and $p_2$ of $f^\perp$ so that the multiplier $q_2$ has a real root distinct from the roots of $p_1$. If this does not hold then we may pick a different set of generators of $f^\perp$: let $p_1'=p_1+\alpha p_2$ with some $\alpha \in \RR$. Then $s=p_1'q_1+p_2(q_2-\alpha q_1)$. We can easily adjust $\alpha$ so that $q_2- \alpha q_1$ has a real root, and we need to argue that we can also make this root distinct from the roots of $p_1'=p_1+\alpha p_2$. Suppose not, then for any $(a.b) \in \RR^2$ that is not a root of $q_1$ we may set $\alpha=-q_2(a,b)/q_1(a,b)$ and make $(a,b)$ a root of $q_2-\alpha q_1$. Therefore we must have $\frac{p_1}{p_2}=-\frac{q_2}{q_1}$ which implies that $s=p_1q_1+p_2q_2=0$ and that is a contradiction. Thus we have $q_2-\alpha q_1=\ell q$ with $\ell \in \RR[x,y]_1$, $q \in \RR[x,y]_{k-m-2}$ and $\ell$ does not divide $p_1'$. Let $I=\langle p_1', \ell p_2 \rangle$. As before, $p_1'$ and $\ell p_2$ form a compete intersection over $\mathbb{C}$ and by Theorem \ref{LEMMA degrees} $I$ is the apolar ideal of some form $g \in \RR[x,y]_{d+1}$. Since $s \in I$ we know that the rank of $g$ is at most $m$ and since $I \subset f^\perp$ we know that the rank of $g$ is at least $m$. Therefore the rank of $g$ is $m$. Further, $g^\perp \subset f^\perp$ has no forms of degree $m-1$ with all real roots and $g^\perp$ is generated in generic degrees. Therefore $g$ is a typical form of rank $m$.

Now suppose that $d=2k+1$ is odd. Then $f^{\perp}$ is generated by forms $p_1,p_2$ with $\deg p_1=k+1$ and  $\deg p_2=k+2$. We note that we only need to deal with the cases $m \geq k+2$. By Apolarity Lemma there exists $s\in (f^\perp)_m$ such that $s$ has all distinct real roots and by Lemma \ref{LEMMA generic} all forms in $(f^\perp)_{m-1}$ have at least $2$ complex roots. Since $s \in (f^\perp)_m$ we can write  $s=p_1q_1+p_2q_2$ for $q_1\in \RR[x,y]_{m-k-1}$ and $q_2 \in\RR[x,y]_{m-k-2}$. The generator $p_1$ is uniquely determined, but $p_2$ is unique only modulo the ideal generated by $p_1$. We now claim that we may choose generators of $f^\perp$ so that the multiplier $q_1$ has a real root distinct from the roots of $p_2$. If this does not hold then let $p_2'=p_2+\ell p_1$ for some linear form $\ell \in \RR[x,y]_1$. We have $s=p_1(q_1-\ell q_2)+p_2'q_2$. We can adjust $\ell$ so that $q_1- \ell q_2$ has a real root, and we need to argue that we may also make this root distinct from the roots of $p_2'=p_2+\ell p_1$. Arguing as before we must have $\frac{p_2}{p_1}=-\frac{q_1}{q_2}$ which implies that $s=p_1q_1+p_2q_2=0$ and that is a contradiction. Let $I=\langle \ell p_1,p_2' \rangle$. Since $\ell p_1$ and $p_2'$ form a compete intersection over $\mathbb{C}$ by Theorem \ref{LEMMA degrees} $I$ is the apolar ideal of some form $g \in \RR[x,y]_{d+1}$. Since $s \in I$ we know that the rank of $g$ is at most $m$ and since $I \subset f^\perp$ we know that the rank of $g$ is at least $m$. Therefore the rank of $g$ is $m$. Further, $g^\perp \subset f^\perp$ has no forms of degree $m-1$ with all real roots and $g^\perp$ is generated in generic degrees. Therefore $g$ is a typical form of rank $m$.
\end{proof}

\end{document}